\documentclass[11pt,a4paper]{amsart}

\usepackage{amscd,amssymb,amsthm}
\usepackage{graphicx}

\numberwithin{equation}{section}

\theoremstyle{plain}
\newtheorem{theorem}{Theorem}[section]

\newtheorem{lemma}[theorem]{Lemma}
\newtheorem{proposition}[theorem]{Proposition}

\theoremstyle{definition}

\theoremstyle{remark}

\numberwithin{equation}{section}

\numberwithin{table}{section}

\numberwithin{figure}{section}

\newcommand{\R}{\mathbb{R}}

\newcommand\mLP{\\[\medskipamount]}

\newcommand\RR{\mathbb{R}}

\newcommand\al\alpha
\newcommand\be\beta
\newcommand\ga\gamma
\newcommand\de\delta
\newcommand\ep\varepsilon
\newcommand\la\lambda

\newcommand\iy\infty

\newcommand{\hyp}[5]{\,\mbox{}_{#1}F_{#2}\!\left(
  \genfrac{}{}{0pt}{}{#3}{#4};#5\right)}

\newcommand{\bea}{\begin{eqnarray*}}
\newcommand{\eea}{\end{eqnarray*}}
\newcommand{\bean}{\begin{eqnarray}}
\newcommand{\eean}{\end{eqnarray}}

\catcode`,\active

\catcode`\,12

\catcode`,\active

\catcode`\,12

\newcommand*\pGqskip{8mu}
\catcode`,\active
\newcommand*\pGq{\begingroup
        \catcode`\,\active
        \def ,{\mskip\pGqskip\relax}%
        \dopGq
}
\catcode`\,12
\def\dopGq#1#2#3#4#5{%
        {}_{#1}F_{#2}\bigg(\genfrac..{0pt}{}{#3}{#4}\,\Big\rvert\,#5\biggr)%
        \endgroup
}
\begin{document}

\title[A generalization of the $2D$ Sleipian functions ]
{A generalization of the $2D$ Sleipian functions}
\author{Fethi BOUZEFFOUR}
\address{Department of Mathematics, College of Sciences\\ King Saud University,
 P. O Box 2455 Riyadh 11451, Saudi Arabia.} \email{fbouzaffour@ksu.edu.sa}


\begin{abstract}The main content of this work is devoted to study various explicit family of special functions generalizing the famous $2D$ Sleipain functions, founded in 1960's by D. Slepian and his co-authors.  As a consequence, many
desirable spectral properties of the corresponding weighted finite Fourier transform are deduced
 from the rich literature. In particular, similar aspect related to Slepian's seminal papers is the investigation of differential operators that commute with appropriate
integral operators are given. Finally, we provided the reader with some analytic expressions for the Fourier transforms of the Disk polynomials and the two variables Gegenbauer polynomials.
 \end{abstract}
 \maketitle
\section{Introduction}In \cite{S-I}, D. Slepian and H.O. Pollak, have proved that the prolate spheroidal wave functions (PSWF's) are eigenfunctions
of both a differential operator
and the finite Fourier transform, and as
such, constitute an orthogonal basis of the space of band--limited functions. In this possibility, prolate spheroidal wave functions naturally occur in fluid dynamics, signal processing,$\dots$. This important property clearly distinguishes PSWF's
from other classes of special functions. Attributed to D. Slepian, the PSWF's are also called as the Slepian functions.\\The generalizations
of the Slepian functions in different directions have subsequently attracted many attentions. For more details on these functions and their computational methods, the reader is referred to \cite{bowkamp,Grunbaum,landau pollak2,S-I,S-A,Landpoll,Moumi1,karoui-moumni2,Niven, Read1,Read2}.\\The multidimensional version of PSWF's can be constructed by redefining
the finite Fourier transform for a compact set of $\R^n$. It should be noted that the theory of Sleipian type function in higher dimensions
raises more interesting properties than in one $1-D$ because of the geometry
that depends on the shape of support of the Fourier transform. In this framework, D. Slepian \cite{S-IV} extended the earlier
works \cite{S-A,S-I,Landpoll} and derived a family of generalized PSWF's from the finite Fourier transform on a unit disk. G. Beylkin et al. \cite{Bey} explored some interesting properties of band--limited functions in a disk. F. Simons et al. \cite{sh} gave
an review and study of time--frequency and time--scale concentration problems on a sphere.\\\indent In this work, we restrict ourselves to the $2D$-case, and we propose a different generalization of the $2D$ Sleipian functions by means of weight function associated with two variables orthogonal polynomials, see \cite{Ko1,Ko2,Krall, Dunkl}. To conform this, we consider the following weight function on the disk of $\R^2$, which is given by $$w_\nu(x,y)=\frac{\nu+1}{\pi} (1-x^2-y^2)^\nu,\quad \nu>-1.$$ By adapting a similar method as \cite{S-IV}, we derive a family of generalized Slepian functions as the eigenfunctions of the weighted finite Fourier transform, and interestingly, they are also the eigenfunctions of a singular Sturm--Liouville problem.\\ \indent The article is organized as follows. In Section 1, we give a brief introduction and we summarize the results that will be needed in the sequel. The main
results are introduced in Section 2, and in Section 3, we conclude the article with some analytic expressions for the Fourier transforms of the Disk polynomials and the two variables Gegenbauer polynomials.
\section{Mathematical preliminaries}
The Jacobi polynomials, defined by
\begin{equation}
\label{Jac_hyp}
P_n^{(\alpha, \beta)}(x) =\binom{\alpha+n}{n}\; \pGq{2}{1}{-n,n+\alpha+\beta+1}{\alpha+1}{\frac{1-x}{2}},
\end{equation}
are orthogonal on $(-1,\,1)$ with respect to the weight
\begin{equation}
(1-x)^{\alpha}(1+x)^{\beta},\quad \alpha,\,\beta>-1.
\end{equation}
The Gegenbauer polynomial $C_n^\nu(x)$ is defined as
\begin{equation}
C_n^\nu(x)=\frac{(2\nu)_n}{(\nu+1/2)_n}P_n^{(\nu-1/2,\nu-1/2)}(x).
\end{equation}
For Bessel functions $J_\nu$ see \cite[Ch. 10]{Wat} and references
given there. We will work with differently normalized Bessel
functions
\begin{align}
&\emph{j}_\nu(x):=\Gamma(\nu+1)\,(2/x)^\nu\,J_\nu(x), \label{4}\\&
\mathcal{J}_\nu(x):=\sqrt{x}\,J_\nu(x)\label{5}.
\end{align}
The Hankel
transform pair takes the form
\begin{equation}
\begin{cases}
&\displaystyle \mathcal{H }_\nu f(y)=\int_{0}^\infty f(x)\,
\mathcal{J}_\nu( xy)\,
\sqrt{x}\,dx,\mLP &\displaystyle
f(x)=\int_{0}^\infty \mathcal{H}_\nu f(y)\,
\mathcal{J}_\nu(xy)\,\sqrt{y}\,dy.
\end{cases}
\label{2}
\end{equation}
The following lemma will be useful in the forthcoming section and the proof partially  deduced from \cite{Abreu }.
\begin{lemma}The finite Hankel transform of the Jacobi polynomials is given
by
\begin{equation}
\int_0^\infty \mathcal{J}_\alpha(xt)\,\chi_{(0,\,1)}\, P_n^{(\alpha,\beta)}(1-2t^2)t^{\alpha+1/2}(1-t^2)^\beta dt=2^{\beta}\frac{\Gamma(\beta+n+1)}{n!}\frac{\mathcal{J}_{\alpha+\beta+2n+1}(x)}{x^{\beta+1}},\end{equation}
where $\chi_{(0,\,1)}$ is the characteristic of the interval $(0,\,1).$ \end{lemma}
\begin{proof}
From the well-known
\begin{align*}&\int_{0}^{\infty}x^{-\lambda }J_{\alpha }(ax)J_{\beta }(bx)\,dx\\&=
\begin{cases}
&\displaystyle  \frac{a^{\lambda -\beta-1 }b^\beta\Gamma(\frac{1}{2}(\alpha+\beta-\lambda+1))}
  {2^{\lambda }\Gamma(\beta +1)\Gamma(\frac{1}{2}(\alpha-\beta+\lambda+1))}
  \hyp21{\frac{1}{2}(\beta-\alpha-\lambda+1),\frac{1}{2}(\alpha+\beta-\lambda+1)}
  {\beta+1}{\frac{b^{2}}{a^{2}}}
  ,b<a\mLP &\displaystyle
  \frac{a^{\alpha }b^{\lambda-\alpha-1}\Gamma(\frac{1}{2}(\alpha+\beta-\lambda+1)}
  {2^{\lambda }\Gamma(\alpha +1)\Gamma(\frac{1}{2}(\beta-\alpha+\lambda+1))}
  \hyp21{\frac{1}{2}(\alpha-\beta-\lambda+1),\frac{1}{2}(\alpha+\beta-\lambda+1)}
  {\al+1}{\frac{a^{2}}{b^{2}}}
  ,a<b,
\end{cases}
\label{2}
\end{align*}
For $0<t<1,$ we put $$a = 1, \quad t = b,\quad
\nu = \alpha,\quad \mu = \alpha+\beta+2n+1,\quad\lambda = \beta.$$ Then from Weber and Schafheitlin integral, one gets
\begin{align*}
 \int_0^\infty x^{-\beta} &J_{\alpha+\beta+2n+1}(x) J_\alpha(xt)\, dx
  = t^{\alpha}\frac{\Gamma(\alpha+n+1)} {2^\beta \Gamma(\alpha+1)\Gamma(\beta+n+1)}\\&\times
  {\,{}_2F_1}(\alpha+n+1,-n-\beta;\alpha+1;t^2),\quad 0<t<1.
\end{align*}
Observe that
\begin{equation*}\hyp21{\alpha+n+1,-n-\beta}
  {\al+1}{t}=(1-t)^\beta\hyp21{-n,
  \alpha+\beta+n+1}
  {\al+1}{t},
\end{equation*}
Hence,
\begin{align*}
 \int_0^\infty x^{-\beta} &J_{\alpha+\beta+2n+1}(x) J_\alpha(xt)\, dx
  = 2^{-\beta} \, \tfrac{\Gamma(n+1)}{\Gamma(\beta+n+1)} \, t^\al(1-t^2)^\beta
  P^{(\alpha,\beta)}_n(1-2t^2), \quad t \in (0,1).
\end{align*}
Now, $t>1$, let us take $$t = a,\quad b = 1, \quad \nu = \alpha+\beta+2n+1,\quad\mu = \alpha, \quad
\lambda = \beta.$$ In this case,
$\frac12(\lambda+\mu-\nu+1) = 0,-1, -2, \dots$, so
$1/\Gamma(\frac12(\lambda+\mu-\nu+1))$ vanishes and we get
$$\int_0^\infty x^{-\beta} J_{\alpha+\beta+2n+1}(x) J_\alpha(xt)\, dx= 0.$$
Therefore
\begin{align*}
 \int_0^\infty x^{-\beta} &J_{\alpha+\beta+2n+1}(x) J_\alpha(xt)\, dx
  = 2^{-\beta} \, \tfrac{\Gamma(n+1)}{\Gamma(\beta+n+1)} \, t^\al(1-t^2)^\beta
  P^{(\alpha,\beta)}_n(1-2t^2)\chi_{(0,1)}, \quad t >0.
\end{align*}The result follows easily from the inversion formula for the Hankel
transform \eqref{2}.
\end{proof}

\section{Generalized $2D$ Sleipian functions}
\subsection{$2D$ Sleipian functions} 
We begin this section by giving some reviews of the properties of the $2D$ Sleipian functions. For a more details see \cite{S-IV}. For $\mathbf{x},\mathbf{y}  \in \RR^2$, we use the usual notation of $\|\mathbf{x}\|$ and $\langle \mathbf{x},\mathbf{y} \rangle$ to denote the Euclidean
norm of $\mathbf{x}$ and the dot product of $\mathbf{x},\mathbf{y}$. Let $c$ be a positive real number, we define the
operator $\mathcal{F}_{c}$ from $L^2(\mathbb{D})$ into itself by
\begin{align}
\mathcal{F}_{c}(f)(\mathbf{\mathbf{x}})=\int_\mathbb{D}  f(\mathbf{y})e^{ic\langle \mathbf{x} ,\mathbf{y}\rangle}d\mathbf{y},\label{b}
\end{align}
where $\mathbb{D}$ is the unit disk of $\R^2.$\\
\indent The spectral analysis of this transform is studied by D. Sleipian, see \cite{S-IV}. The integral operator \eqref{bb} is a compact operator with a non--degenerate kernel, then it has an infinite and countable set of eigenfunctions. We denote by $\psi_{n,c}(\mathbf{x})$ the eigenfunction
that corresponds to eigenvalue $\mu_n$, so that
\begin{align}
\int_\mathbb{D}  \psi_{n,c}(\mathbf{x})e^{ic\langle \mathbf{x} ,\mathbf{y}\rangle}d\mathbf{y}=\mu_n\psi_{n,c}(\mathbf{x}).\label{bb}
\end{align}
The eigenfunctions $\{\psi_{n,c}(\mathbf{x})\}$ are orthogonal
both on $\mathbb{D}$ and on $\R^2$, and are complete in $L^2(\mathbb{D})$. They called $2D$ Sleipian functions. For an integer $N\geq0$, the function  $\psi_{n,c}(\mathbf{x})$ is written as a product an angular function of the form $\cos( N\theta) $ or $\sin( N \theta) $ and a radial function $\varphi_{N,n,c}(x)$  satisfying the following
eigenproblem
\begin{align}
\mathcal{H}_c(\varphi_{N,n,c})(x)=\mu_{N,n}(c)\,\varphi_{N,n,c}(x)\label{Ker},
\end{align}
where the finite Hankel transform  $\mathcal{H}_c(\varphi)$ is defined by
 \begin{align}
\mathcal{H}_c(\varphi)(x)=\int_{0}^1J_N(cxt)
\sqrt{cxt}\,\varphi(t)dt.\label{hankl}
\end{align}
In \cite{S-IV}, the author has proved that the following differential operator
\begin{align}
(L_\nu y)(t)=(1-t^2)y''(t) -2ty'(t)+(\frac{\frac{1}{4}-N^2}{t^2}-c^2t^2)y(t),\label{daff1}\end{align}
 commutes with the finite Hankel transform \eqref{hankl}. Many desirable properties, computational schemes, asymptotic results and expansions of the $2D$ Sleipain are consequences of the previous commutativity property.
\subsection{Generalized $2D$ Sleipian functions}For given $\nu>-1,$ we consider the weight function $w_\nu(\mathbf{x})$ on $\mathbb{D}$:
\begin{equation}
w_\nu(\mathbf{\mathbf{x}})=\frac{\nu+1}{\pi}(1-x^2-y^2)^\nu,\quad  \mathbf{x}=(x,\,y).\label{weight1}
\end{equation}
Let $L^2_\nu(\mathbb{D})$ denote the Hilbert space of all square integrable functions on  $\mathbb{D}$ with respect to the weight $w_\nu$, and equipped with the inner product
\begin{equation}
\langle f,g\rangle_\nu=\int_{\mathbb{D}}f(\mathbf{x})\overline{g(\mathbf{x})} w_\nu(\mathbf{x}) \,d\mathbf{x}.
\end{equation}
Define the operator $\mathcal{F}_{\nu,c}$  from $L^2_\nu(\mathbb{D})$ into its self by
\begin{align}
\mathcal{F}_{\nu,c}(f)(\mathbf{\mathbf{x}})=\int_{\mathbb{D}}
e^{ic\langle \mathbf{x},\mathbf{y}\rangle}f(\mathbf{y}) w_\nu (\mathbf{y})d\mathbf{y}\label{Fourier}.
\end{align}
By standard arguments one can show that the finite weighted Fourier transform  $\mathcal{F}_{\nu,c}$
defined in \eqref{Fourier} is a compact operator on $L^2_\nu(\mathbb{D})$, then it has a sequence of eigenvalues $\{\lambda_n\}_{n=0}^\infty$ satisfying
$$|\lambda_0|\geq \dots |\lambda_n|\geq\dots \geq 0.$$ The corresponding eigenfunctions are denoted by $\{\psi^\nu_{n,c}\}_n$ and will be called $2D$ Sleipian functions of order $\nu,$ so that
\begin{equation}
 \int_{\mathbb{D}}
e^{ic\langle \mathbf{x},\mathbf{y}\rangle}\psi^\nu_{n,c}(\mathbf{x})w_\nu(\mathbf{x})d\mathbf{x}
 =\lambda_n\psi^\nu_{n,c}(\mathbf{y}).\label{es1}
 \end{equation}
Since, $\mathcal{F}_{\nu,c}$ is injective and compact it follows that $\{\psi^\nu_{n,c}\}$ form an orthogonal basis of $L^2_\nu(\mathbb{D})$. We normalize the eigenfunctions so that
\begin{equation*}
 \int_{\mathbb{D}}
\psi^\nu_{n,c}(\mathbf{x})\psi^\nu_{m,c}(\mathbf{x})w_\nu(\mathbf{x})d\mathbf{x}
 =\delta_{n,m}.
 \end{equation*}
Note that for $\nu=0,$ $\mathcal{F}_{\nu,c}$ is reduced to the finite Fourier transform on the unit disk, which is given by
\begin{align}
\mathcal{F}_{c}(f)(\mathbf{\mathbf{x}})=\int_\mathbb{D}  f(\mathbf{y})e^{ic\langle \mathbf{x} ,\mathbf{y}\rangle}d\mathbf{y}.\label{b}
\end{align}
We denote by $\mathcal{F}^*_{\nu,c}$ its adjoint, in other words
\begin{equation}
\langle \mathcal{F}_{\nu,c}(f),\,g\rangle_\nu=\langle f,\,\mathcal{F}^*_{\nu,c}(g)\rangle_\nu.
\end{equation}
A simple calculation shows that
\begin{align}
\mathcal{F}^*_{\nu,c}(f)(\mathbf{\mathbf{y}})=\int_{\mathbb{D}}
e^{-ic\langle \mathbf{x},\mathbf{y}\rangle}f(\mathbf{x})w_\nu(\mathbf{x})d\mathbf{x}\label{Fourier2}.
\end{align}
Let consider the following operator $\mathcal{K}_{\nu,c}$, given by \begin{align}
 \mathcal{K}_{\nu,c}=\mathcal{F}_{\nu,c}\circ\mathcal{F}^*_{\nu,c},\quad \nu,\,\mu>-1.\label{hob2}
\end{align}
\begin{proposition}Let $c>0$, $\nu>-1$ and $f\in L^2_\nu(\mathbb{D}),$ we have
\begin{equation}
(\mathcal{K}_{\nu,c}f)(\mathbf{\mathbf{y}})=\int_{\mathbb{D}}f(\mathbf{z})\emph{j}_{\nu+1}
(c\|\mathbf{y}-\mathbf{z}\|)w_\nu(\mathbf{z})d\mathbf{z}.
\end{equation}
\end{proposition}
\begin{proof}
From \eqref{hob2} and \eqref{Fourier2} it follows that
\begin{equation}
(\mathcal{K}_{\nu,c} f)\mathbf{y})=\int_{\mathbb{D}}K(\mathbf{y},\mathbf{z})
f(\mathbf{z})w_\nu(\mathbf{z})d\mathbf{z}
\end{equation}
where
\begin{equation}
K(\mathbf{y},\mathbf{z})=\int_{\mathbb{D}}e^{ic(\langle \mathbf{x},\,\mathbf{y}-\mathbf{z}\rangle}w_\nu(\mathbf{x})d\mathbf{x}.
\end{equation}
Setting \begin{align} \mathbf{y}-\mathbf{z}=(\varrho \cos\vartheta,\,\varrho \sin\vartheta),\quad \varrho=\|\mathbf{y}-\mathbf{z}\|,
\end{align}
and we convert the last integral to polar coordinates, we get
\begin{align}
K(\mathbf{y},\mathbf{z})&=\frac{\nu+1}{\pi}\int_0^1\int_0^{2\pi}e^{icr\varrho \cos(\phi-\theta)}r(1-r^2)^\nu drd\theta\\&=
2(\nu+1)\int_0^1J_0(cr\varrho )r(1-r^2)^\nu dr.
\end{align}
From the Sonine integral for Bessel function \cite[\S12.11]{Wat}
\begin{equation}
J_{\lambda+\alpha}(y)=\frac{y^\lambda}{2^{\lambda-1}\Gamma(\lambda)}
\int_0^1x^{\alpha+1}(1-x^2)^{\lambda-1}J_\alpha(xy)dx,\,\,\, \Re(\alpha)>-1,\,\,\, \Re(\lambda)>0.
\end{equation}
We obtain
\begin{equation}
K(y,z)=\Gamma(\nu+2)(2/cr)^{\nu+1}J_{\nu+1}(c\rho)= \emph{j}_{\nu+1}(c\|\mathbf{y}-\mathbf{z}\|).
\end{equation}
\end{proof}
By iterating \eqref{es1}, one finds that the function $\psi^\nu_{n,c}(\mathbf{x})$, also satisfy
\begin{equation}
(\mathcal{K}_{\nu,c}\psi^\nu_{n,c})(\mathbf{x})=\int_{\mathbb{D}}\psi^\nu_{n,c}(\mathbf{x})\emph{j}_{\nu+1}
(c\|\mathbf{x}-\mathbf{y}\|)w_\nu(\mathbf{y})d\mathbf{z}=|\lambda_n|\psi^\nu_{n,c}(\mathbf{x}).
\end{equation}
\subsection{Computation of the generalized $2D$ Sleipian functions}
We now give some details related to the integral equation
\begin{equation}
 \int_{\mathbb{D}}
e^{ic\langle \mathbf{x},\mathbf{y}\rangle}\psi^\nu_{n,c}(\mathbf{x})w_\nu(\mathbf{x})d\mathbf{x}
 =\lambda_n\psi^\nu_{n,c}(\mathbf{y}).\label{es2}
 \end{equation}
 The computation of the generalized $2D$ Sleipain functions has an extra difficulty
due to the evaluation of the eigenvalues of an integral operator defined by a double integral. To overcome this difficulty, we convert the first member of the integral equation \eqref{es2}  to polar coordinates, to get
\begin{align}
\frac{\nu+1}{\pi}\int_{0}^1\rho(1-\rho^2)^\nu\,d\rho\int _0^{2\pi}e^{icr\rho\cos(\theta-\vartheta) }\psi^\nu_{n,c}(\rho,\vartheta)d\vartheta,
\end{align}
and then, we use the relation
\begin{equation}
e^{ir\cos \theta}=\sum_{n=-\infty}^\infty i^ne^{in\theta}J_n(r)
\end{equation}
to obtain
\begin{align}
\frac{\nu+1}{\pi}\sum_{n=-\infty}^\infty i^ne^{in\theta}\int_{0}^1\rho(1-\rho^2)^\nu J_n(cr\rho)\,d\rho\int _0^{2\pi}e^{-in\vartheta }\psi^\nu_{n,c}(\rho,\vartheta)d\vartheta.
\end{align}
It is easy to see that the eigenfunctions and eigenvalues of \eqref{es2} may be given by
\begin{align}
 &\psi_{N,n}^{(\nu,c)}(r,\theta)=R^{(\nu,c)}_{N,n}(r)e^{iN\theta},
 \end{align}
where the radial part $R^{(\nu,c)}_{N,n}(r)$ is a solution of the following integral equation\begin{align}
\int_{0}^1J_N(cr\rho)
R^{(\nu,c)}_{N,n}(\rho)\rho(1-\rho^2)^\nu\,d\rho
=\mu_{N,n} R^{(\nu,c)}_{N,n}(r).
\end{align}
and
\begin{equation}
 \lambda_{N,n}=2(\nu+1)i^N\mu_{N,n},\,\,n,\,N=0,\,1,\,2,\,\dots,\,\,.
\end{equation}

Let denote by $L^2_\nu(0,1)$ the Hilbert space  $L^2((0,1),(1-t^2)^\nu dt)$, equipped with the inner product and norm
\begin{equation}
( f,g)_\nu=\int_0^1f(t)g(t)(1-t^2)^\nu\,dt,\quad \|f\|_\nu=\sqrt{( f,g)_\nu},
\end{equation}
and consider the operator
\begin{equation}
\mathcal{H}_{c,N,\nu}(f)(x)=\int_{0}^1\mathcal{J}_N(cxt)
f(t)\,(1-t^2)^\nu dt, \quad f\in L^2_\nu(0,1).\label{hankel}
\end{equation}
Obviously, $\mathcal{H}_{c,N}$ is compact and self--adjoint. Moreover, the eigenvalues of $\mathcal{H}_{c,N}$ are
$\{\sqrt{c}\mu_{N,n}\}_{n=0}^\infty,$ and the correspond eigenfunction \begin{equation}\varphi_{N,n}(x)=\sqrt{x}R^{(\nu,c)}_{N,n}(x).\label{prolate2}\end{equation}
Define the differential operator  $L_{c,N,\nu}$
 \begin{align}
 L_{c,N,\nu}&=(1-x^2)\frac{d^2}{dx^2}-2(\nu+1)x\frac{d}{dx} +\frac{\frac{1}{4}-N^2}{x^2}-c^2x^2.\label{op1}\end{align}

 \begin{theorem}
Suppose that $c > 0$ and  $f(x)\in \mathcal{C}^2((0,1))$ such that $f(0)=1$. Then
\begin{align}
\mathcal{H}_{c,N}\circ L_{c,n,\nu}(f)(x)=L_{c,N,\nu}\circ\mathcal{H}_{c,N}(f)(x)\label{ha1}.
\end{align}
\end{theorem}
\begin{proof}From \eqref{hankel} and \eqref{op1}, we have
\begin{align*}
\mathcal{H}_{c,N}(L_{c,N,\nu}f)(x)&=\frac{\nu+1}{\pi}\int_{0}^1\mathcal{J}_N(cxt)
\big[(1-t^2)^{\nu+1}f'(t)]'\\&+(\frac{\frac{1}{4}-N^2}{t^2}-c^2t^2)f(t)\,(1-t^2)^{\nu}
\big]dt.
\end{align*}
Then
\begin{equation*}
\mathcal{H}_{c,N}(L_{c,N,\nu}f)(x)=I+J
\end{equation*}
where \begin{equation*}
I=\frac{\nu+1}{\pi}\int_{0}^1\mathcal{J}_N(cxt)
[(1-t^2)^{\nu+1}f'(t)]'dt,
\end{equation*}
and
\begin{equation*}
J=\frac{\nu+1}{\pi}\int_{0}^1\mathcal{J}_N(cxt)(\frac{\frac{1}{4}-N^2}{t^2}-c^2t^2)f(t)\,(1-t^2)^{\nu}
dt.
\end{equation*}
By applying integration by parts formula twice, the integral $I$ becomes
\begin{align}
I=&\frac{\nu+1}{\pi}\big[(1-t^2)^{\nu+1}\big(\mathcal{J}_N(cxt)f'(t)-cx\mathcal{J}'_N
(cxt)f(t)\big)\big]_0^1\label{int1}\\&+\frac{\nu+1}{\pi}\int_{0}^1\big[c^2x^2
(1-t^2)
\mathcal{J}''_N(cxt)-2cxt(\nu+1)\mathcal{J}'_N(cxt)
\big]f(t)(1-t^2)^{\nu}dt.\nonumber
\end{align}
Since $$\mathcal{J}_N(0)=f(0)=0$$
and \begin{equation}
\mathcal{J}''_N(cxt)=-(1+\frac{\frac{1}{4}-N^2}{c^2x^2t^2})\mathcal{J}_N(cxt).\label{Bessel2}
\end{equation}
Substitute this expression in \eqref{int1} to yield
\begin{align}
I=\frac{\nu+1}{\pi}\int_{0}^1\big[-c^2x^2
(1-t^2)(1+\frac{\frac{1}{4}-N^2}{c^2x^2t^2})
\mathcal{J}_N(cxt)-2cxt(\nu+1)\mathcal{J}'_N(cxt)
\big]f(t)(1-t^2)^{\nu}dt.\nonumber
\end{align}
Thus, by adding $I$ and $J$ we have
\begin{align}
\mathcal{H}_{c,N}(L_{c,N,\nu}f)(x)=&\frac{\nu+1}{\pi}
\int_{0}^1\big[\frac{1}{4}-N^2-c^2x^2
-c^2t^2+c^2x^2t^2)
\mathcal{J}_N(cxt)\label{int3}\\&-2cxt(\nu+1)\mathcal{J}'_N(cxt)
\big]f(t)(1-t^2)^{\nu}dt.\nonumber
\end{align}
On the other hand, by direct computations and use of \eqref{Bessel2}, one has
\begin{align}\label{int2}
L_{c,N,\nu}\circ\mathcal{H}_{c,N}(f)(x)&=\frac{\nu+1}{\pi}\int_{0}^1
\big[(1-x^2)^{-\nu}[ct(1-x^2)^{\nu+1}\mathcal{J}'_N(cxt)]'\\&+
(\frac{\frac{1}{4}-N^2}{x^2}-c^2x^2+c^2x^2t^2)\mathcal{J}_N(cxt)\big]f(t)\,(1-t^2)^{\nu}
dt\nonumber\\&=\frac{\nu+1}{\pi}\int_{0}^1\big[\frac{1}{4}-N^2-c^2x^2
-ct^2)
\mathcal{J}_N(cxt)\nonumber\\&-2cxt(\nu+1)\mathcal{J}'_N(cxt)
\big]f(t)(1-t^2)^{\nu}dt.\nonumber
\end{align}
By comparing \eqref{int2} and \eqref{int3}, we conclude that $\mathcal{H}_{c,N}$ and $L_{c,N,\nu}$ commute.
\end{proof}
The unbounded operator $(L_{c,N,\nu}
,\mathcal{D}_0)$  with domain
$D_0=\mathcal{C}^2(0,1)$ is positif and is essentially self--adjoint. Then there exists a strictly increasing unbounded sequence of positive numbers
$\{\chi_{N,\nu,n}\}_{n=0}^\infty$ such that for each nonnegative integer $n$, the differential equation
\begin{equation}
L_{c,N,\nu}\varphi(x)=\chi_{N,\nu,n}\varphi(x),\label{equation}
\end{equation}has a solution which is bounded on the interval
$(0\,,1)$ and such that $\varphi(0)=0.$
Moreover, the Theorem 4.1 implies that $L_{c,N,\nu}$ and $\mathcal{H}_{c,N}$ have the same eigenvectors. It follows that the functions $\varphi_{N,n}(x)$ defined in  \eqref{prolate2} are the eigenfunctions of $L_{c,N,\nu}$.\\ \indent
In the case $c=0$, the differential equation \eqref{equation} becomes  \begin{align}
(1-x^2)y''(x) -2(\nu+1)xy'(x)+(\frac{\frac{1}{4}-N^2}{x^2})y(x)=\chi(0)y(x).\label{dana}\end{align}
Making use of the substitution $y(x)=t^{N+\frac{1}{2}}R(x)$, we obtain
\begin{align}(1-x^2)R''(x)&+\big[\frac{2N+1}{x}-(2N+2\nu+3)x\big]\nonumber\\&\times
R'(x)-(N^2-\frac{1}{4}+(2N+1))(\nu+1)R(x)
=\chi(0)R(x).\label{zaf}
\end{align}
From the differential equation of the Jacobi polynomials, we see that the equation \eqref{zaf} possesses an infinite family of polynomial eigenfunctions $\{R_{n}^{(\alpha,\beta)}(x)\}_{n=0}^{\infty}$, where $R_{n}^{(\nu)}(x)$ is expressed as follows
\begin{equation}R_{N,\,n}^{(\nu)}(x)=\frac{N!n!}{(n+N)!}\,P_n^{(N,\nu)}(1-2x^2).
\end{equation}
Then the solution of \eqref{dana} is \begin{equation}
T^\nu_{N,\,n}(x)=x^{N+\frac{1}{2}}R_{N,\,n}^{(\nu)}(x),
\end{equation}
and its corresponding eigenvalue $\chi_{N,\,n}(0)$ is given by
\begin{equation}
\chi_{N,\,n}(0)=(N+2n+\frac{1}{2})(N+2\nu+2n+\frac{3}{2}).
\end{equation}
From the orthogonality relations and recurrence relation for the Jacobi polynomials, we have
\begin{equation}
\int_0^1T^\nu_{N,\,n}(x)T^\nu_{N,\,n}(x)(1-x^2)^\nu\,dx=\frac{N!\Gamma(n+\nu+1)}
{(2n+N+\nu+1)\Gamma(n+N+\nu+1)}\delta_{nm}.
\end{equation}
and
\begin{align}
x^2T^\nu_{N,\,n}(x)&=a_n
T^\nu_{N,\,n+1}(x)+b_n T^\nu_{N,\,n}(x)+c_nT^\nu_{N,\,n-1}(x),
\end{align}
where
\begin{align*}
&a_n=-\frac{(n+N+\nu+1)^2}{(2n+N+\nu+1)(2n+N+\nu+2)}\\&
b_n=\frac{1}{2}\big(1+\frac{N^2-\nu^2}{(2n+N+\nu)
(2n+N+\nu+2)}\big)\\&
c_n=-\frac{n(n+\nu)}{(2n+N+\nu)(2n+N+\nu+2)}.
\end{align*}
To compute $\varphi_{N,n}(x)$ , we use the same technique that has been used in [6] for the computation of the prolate spheroidal
wave functions. This technique was also used in [25] for the computation of the circular prolate spheroidal wave functions.
The Fourier series expansion of ƒ³n,c with respect to {hn, n ¸ N}, is given by
\begin{equation}
\varphi_{N,n}(x)=\sum_{k=0}^\infty A^{(N,\nu)}_k(c)T^\nu_{N,\,k}(x).\label{ser}
\end{equation}
Substitution in \eqref{ser} yields the three-term recurrence
\begin{align}
&c^2a_{k-1}A^{(N,\nu)}_{k-1}+\big[c^2b_k+(N+2n+\frac{1}{2})
(N+2\nu+2n+\frac{3}{2})-\chi\big]A^{(N,\nu)}_{k}+c^2c_kA^{(N,\nu)}_{k+1}=0.
\end{align}
This recurrence can be used to determine the $A^{(N,\nu)}_{k}$ and the eigenvalues
in a manner quite parallel to that used in \cite{S-IV}.

\section{The finite Fourier transform and $2D$ orthogonal polynomials}
In this section, we give a variety of formulas for the finite Fourier transform of the disk polynomials and the two variables Gegenbauer polynomials.\\ \indent
The disk polynomials $\{D^\nu_{n,\,m}(x,y)\}$ are defined in terms of the Jacobi polynomials $\{P_n^{(\alpha,\beta)}(x)\}_n$ as  \cite{Dunkl,Ko1} \begin{equation}\label{disk-poly2}
   D^\nu_{n,\,m}(\mathbf{x}) =
   \frac{(-1)^{n\wedge m}m!}{(\nu+1)_m}\, e^{i(n-m)\theta} P_{n\wedge m}^{(|m-n|,\nu)}(1-2r^2),
\end{equation}
where $$\mathbf{x}=(r\cos\theta,r\sin\theta).$$
They satisfy the orthogonality relations,
\begin{equation}
\int_{\mathbb{D}}D_{m,\,n}^\nu(\mathbf{x})\overline{D_{l,\,k}^\nu(\mathbf{x})}w_\nu(\mathbf{x})
d\mathbf{x}=(\pi^\nu_{m,\,n})^{-1}\delta_{m,\,l}\delta_{n,\,k},
\end{equation}
where
\begin{equation}
\pi^\nu_{m,\,n}=\frac{m+n+\nu+1}{\nu+1}\frac{(\nu+1)_n(\nu+1)_m}{n!m!}.
\end{equation}
Another orthogonal basis of  $L^2_\nu(\mathbb{D})$ $(\nu>-1)$ is given in terms of the two variables Gegenbauer polynomials \cite[\S2.3]{Dunkl}
\begin{align}P^\nu_{n,k}(x,y)=
  C_{n-k}^{\nu+k+1/2}(x)(1-x^2)^{k/2}C_{k}^{\nu}(\frac{x}
 {\sqrt{1-y^2}}),\quad  \nu \neq 0,
\end{align}
where $0\leq k\leq n$ and  $-1<x<1.$\\ \indent
The orthogonality relation \begin{equation}
\int_{\mathbb{D}}P^\nu_{n,k}(\mathbf{x})P^\nu_{m,l}(\mathbf{x})w_{\nu}(\mathbf{x})
d\mathbf{x}=h^\nu_{n,\,k}\delta_{n,\,m}\delta_{k,\,l},
\end{equation}
where
\begin{align}h^\nu_{n,\,k}=
\frac{(2k+2\nu+1)_{n-k}(2\nu)_k(\nu)_k(\nu+1/2)}{(n-k)!k!(\nu+1/2)_k(n+\nu+1/2)}.
\end{align}

In the following theorem, we compute the image by the $2D$ finite weighted Fourier transform of the disk polynomials.
\begin{theorem}The finite Fourier transform of the disk polynomial is given
by  \\ \begin{align*}
\int_{\mathbb{D}}D_{m,\,n}^\nu(\mathbf{x})e^{i \langle\mathrm{x},\mathrm{y}\rangle}w_\nu(\mathbf{x})
d\mathbf{x}=c_{n,m}(\nu)\,
(2/\rho)^{\nu+1}J_{\nu+n+ m +1}(\rho)e^{i(m-n)\vartheta}, \quad \mathbf{y}=(\rho \cos\vartheta,\,\rho\sin\vartheta),
\end{align*}
where
$$c_{n,m}(\nu)=(-1)^m(\nu+1)i^{n-m}\frac{\Gamma(n\wedge m+1)}{\Gamma(\nu+n\wedge m+1)}.$$
\end{theorem}
\begin{proof}

Let $\mathbf{y}=(\rho \cos\vartheta,\rho \sin\vartheta)$
\begin{align*}
\int_{\mathbb{D}}D_{m,\,n}^\nu(\mathbf{x})e^{i \langle\mathrm{x}.\mathrm{y}\rangle}w_\nu(\mathbf{x})
d\mathbf{x}&=(-1)^{n\wedge m}\frac{\nu+1}{\pi}\int_0^1\int_0^{2\pi}
e^{ir\rho\cos(\theta-\vartheta)}P_{n\wedge m}^{(|m-n|,\nu)}(1-2r^2)\nonumber\\&\times r^{|m-n|+1}e^{i(m-n)\theta}\,(1-r^2)^{\nu}\,drd\theta.
\end{align*}
Observe that
\begin{equation*}
e^{ir\cos \theta}=\sum_{n=-\infty}^\infty i^ne^{in\theta}J_n(r).
\end{equation*}
Then
\begin{equation}
\int_0^{2\pi}e^{icr\rho \cos(\theta-\vartheta)}e^{i(m-n)\theta}\,d\theta=2\pi i^{n-m}e^{i(m-n)\vartheta}J_{n-m}(r\rho),
\end{equation}
and
\begin{align}
\int_{\mathbb{D}}D_{m,\,n}^\nu(\mathbf{x})e^{i \langle\mathrm{x}.\mathrm{y}\rangle}w_\nu(\mathbf{x})
d\mathbf{x}&=2(-1)^{n\wedge m}(\nu+1)i^{n-m}e^{i(m-n)\vartheta}\int_0^1
J_{n-m}(r\rho)\,P_{n\wedge m}^{(\nu,|m-n|)}(1-2r^2)\nonumber\\&\times r^{|m-n|+1}\,(1-r^2)^{\nu}\,dr.
\end{align}
From lemma , we obtain
\begin{align}
\int_{\mathbb{D}}D_{m,\,n}^\nu(\mathbf{x})e^{i \langle\mathrm{x}.\mathrm{y}\rangle}w_\nu(\mathbf{x})
d\mathbf{x}
=(-1)^m(\nu+1)i^{n-m}\frac{\Gamma(n\wedge m+1)}{\Gamma(\nu+n\wedge m+1)}
(2/\rho)^{\nu+1}J_{\nu+n+ m +1}(\rho)e^{i(m-n)\vartheta},
\end{align}
\end{proof}
\begin{theorem}The finite Fourier transform of the $2D$ Gegenbauer polynomial is given
by  \begin{equation}
\int_{\mathbb{D}}P^{\nu+1/2}_{n,k}(\mathbf{x})e^{i \langle\mathrm{x},\mathrm{y}\rangle}w_\nu(\mathbf{x})
d\mathbf{x}=\zeta_{n,m}(\nu)\rho^{k-n}J_{\nu+n+1}(\rho) C_{n-k}^{\nu+k+1}(\cos\vartheta),
\end{equation}
where
$$\zeta_{n,m}(\nu)=2^{\nu+1}\Gamma(\nu+1)\pi\frac{(-1)^n
(2\nu+1)_n}{i^k(2n)!}.$$
\end{theorem}
\begin{proof}
Observe that
\begin{align*}
\int_{\mathbb{D}}f(x,y)dxdy&=\int_{-1}^1\int_{-\sqrt{1-x^2}}^
{\sqrt{1-x^2}}f(x,y)\,dxdy,\\&
=\int_{-1}^1\int_{-1}^{1}f(x,t\sqrt{1-x^2})\,\sqrt{1-x^2}dtdx\\&
=\int_{0}^{\pi}\int_{0}^{\pi}f(\cos\theta,\cos \vartheta\sin\theta)\,\sin^2 \theta\sin\vartheta d\vartheta d\theta.
\end{align*}
Let $\mathbf{y}=(\rho \cos \phi,\,\rho \sin\phi),$ then for $\nu>-1/2$ and $\nu\neq 0$, we have
\begin{align}
\int_{\mathbb{D}}P^{\nu+1/2}_{n,k}(\mathbf{x})e^{i \langle\mathrm{x},\mathrm{y}\rangle}w_\nu(\mathbf{x})
d\mathbf{x}&=\int_0^\pi\int_0^\pi e^{i\rho(\cos \theta \cos \phi+\sin\theta \cos\vartheta \sin\phi)}C_{n-k}^{\nu+k+1}(\cos\theta)\\&\times
C_{k}^{\nu+1/2}(\cos\vartheta)\sin^{2\nu+k+2}\theta \sin^{2\nu+1}\vartheta d\vartheta d\theta.
\end{align}From Gegenbauer's generalization of Poisson's
integral in \cite{Wat}, we obtain
\begin{equation}
J_{\nu+n+1}(x)=\frac{(-i)^nn!(x/2)^\nu}
{\Gamma(\nu+1/2)\Gamma(1/2)(2\nu)_n}\int_{0}^\pi e^{ix\cos \vartheta}C_n^\nu(\cos\vartheta)\sin^{2\nu}\vartheta\,d\vartheta
\end{equation}
\begin{align}
\int_{\mathbb{D}}P^{\nu+1/2}_{n,k}(\mathbf{x})e^{i \langle\mathrm{x},\mathrm{y}\rangle}w_\nu(\mathbf{x})
d\mathbf{x}
&=\frac{2^{\nu+1/2}\Gamma(\nu+1)
\Gamma(1/2)(2\nu+1)_n}{(-i)^n(2n)!}\int_{0}^\pi e^{i\rho \cos\theta \cos\vartheta}\frac{J_{\nu+k+1/2}(\rho \sin\theta \sin\phi)}{(\rho \sin\theta \sin\phi)^{\nu+k+1/2}}\\&\times C_{n-k}^{\nu+k+1}(\cos\theta)
\sin^{2\nu+2k+2}\theta\,d\theta.
\end{align}
From Gegenbauer's finite integral \cite[\S12.14]{Wat}
\begin{align}
\int_0^\pi \frac{J_{\nu-1/2}(r \sin\theta \sin\vartheta)}{(r\sin \theta \sin \vartheta)^{\nu-1/2}} e^{ir \cos\theta \cos\vartheta}\,C_n^{\nu}(\cos\theta)\sin^{2\nu}(\theta)\,d \theta=
\sqrt{2\pi}i^n \frac{J_{\nu+n}(r)}{r^\nu}C_n^{\nu}(\cos\vartheta),
\end{align}
we obtain
\begin{align}
\int_{\mathbb{D}}P^{\nu+1/2}_{n,k}(\mathbf{x})e^{i \langle\mathrm{x},\mathrm{y}\rangle}w_\nu(\mathbf{x})
d\mathbf{x}
&=2^{\nu+1}\Gamma(\nu+1)\pi\frac{(-1)^n
(2\nu+1)_n}{i^k(2n)!}\rho^{k-n}J_{\nu+n+1}(\rho) C_{n-k}^{\nu+k+1}(\cos\phi).
\end{align}
\end{proof}

\end{document}